\documentclass[11pt]{article}

\usepackage[a4paper,margin=1in]{geometry}
\usepackage{amsmath,amssymb,amsthm,mathtools}
\usepackage[hidelinks]{hyperref}
\usepackage{enumitem}
\usepackage{microtype}
\usepackage{booktabs}
\usepackage{array}
\usepackage{libertinus}
\newtheorem{theorem}{Theorem}[section]
\newtheorem{proposition}[theorem]{Proposition}
\newtheorem{lemma}[theorem]{Lemma}
\newtheorem{corollary}[theorem]{Corollary}

\newtheorem{definition}[theorem]{Definition}

\newcommand{\E}{\mathbb E}
\newcommand{\Z}{\mathbb Z}
\newcommand{\R}{\mathbb R}
\newcommand{\C}{\mathbb C}
\newcommand{\1}{\mathbf 1}
\newcommand{\dist}{\operatorname{dist}}
\newcommand{\Ldim}{\operatorname{Ldim}}
\newcommand{\eqc}{\operatorname{eq}}
\newcommand{\PerR}{\operatorname{Per}_{R}}
\newcommand{\Dred}{D^{\sharp}}

\title{Blocky Matrices and Group Idempotents}
\author{Gaia Carenini \footnote{Department of Pure Mathematics and Mathematical Statistics, University of Cambridge, Cambridge.}}
\date{\today}

\begin{document}
\maketitle

\begin{abstract}
We prove a common generalization of two structure theorems: the dimension-free decomposition theorem for idempotent Schur multipliers and the idempotent theorem in harmonic analysis. Roughly speaking, our result shows that an invariant integer-valued kernel with Hilbert-space factorization norm $\gamma$ admits a signed decomposition into at most $2^{O(\gamma^4)}$ elementary pieces. In the matrix setting these pieces are blocky matrices, while in the group setting they are indicators of cosets. In the locally compact abelian setting this quantitatively strengthens the theorem of Green and Sanders, while in the non-abelian setting it gives a quantitative strengthening of Host's idempotent theorem and, for finite groups, of Sanders's quantitative result. It also improves the exponent in the dimension-free matrix decomposition from $\gamma^6$ to $\gamma^4$. 
\end{abstract}

\section{Introduction}

\subsection{A common decomposition problem}

We begin with two pieces of mathematics that at first sight belong to
different subjects, with the aim of making clear why the same decomposition
principle should govern both.

The first comes from the study of Boolean matrices and communication
complexity. Call a $0/1$ matrix $B$, with rows indexed by a set $X$ and
columns by a set $Y$, \emph{blocky} if, after independently permuting its rows
and its columns, it becomes a block-diagonal matrix each of whose blocks is
entirely $1$s, together with some all-zero rows and columns left over. A
moment's thought shows that this is the same as saying that there is a
partition of (a subset of) $X$ into classes $X_\lambda$, a partition of (a
subset of) $Y$ into classes $Y_\lambda$ indexed by the \emph{same} label set,
and that $B(x,y)=1$ exactly when $x$ and $y$ have been placed in the same
class. In other words, if we define two labelling functions $\phi,\psi$,
sending each row and column either to its class label or to a dummy symbol
$\bot$ for the leftover rows and columns, then
$$
B(x,y) = 1 \iff \phi(x) = \psi(y) \in \Lambda .
$$
So a blocky matrix is nothing but an \emph{equality test}, dressed up by first
relabelling the rows and columns. This is a very simple idea, but it is worth
stating carefully, because it is about to reappear in a completely different
guise.

The second piece of mathematics comes from harmonic analysis. Let $G$ be a
group and let $H \le G$ be a subgroup, and consider a left coset $aH$. Its
indicator function, viewed through the translation-invariant kernel
$$
M(x,y):=\1_{aH}(x^{-1}y),
$$
satisfies
$$
\1_{aH}(x^{-1}y)=1
\iff x^{-1}y\in aH
\iff y\in xaH
\iff (xa)H=yH .
$$
Look at the right-hand side: it says that the coset $(xa)H$, viewed as a
function of $x$, equals the coset $yH$, viewed as a function of $y$. That is
exactly an equality test again, except that now the ``labels'' are cosets of
$H$ in the quotient set $G/H$, rather than an abstract set $\Lambda$ chosen ad
hoc.

So both a blocky matrix and a coset kernel are equality kernels; the only
difference is what plays the role of the label set, and, crucially, whether
there is a group acting compatibly on everything. This suggests that
\emph{both} of the structure theorems built from these atoms -- the
dimension-free theorem of Beke, Goh, Hatami, Jaffe and Naylor \cite{BGHJN}
for Schur multipliers, the quantitative idempotent theorem of Green and
Sanders in the abelian setting \cite{GS}, Host's non-abelian idempotent
theorem \cite{Host}, and Sanders's quantitative finite-group refinement
\cite{SandersNA} -- ought to be instances of one master statement: an integer-valued kernel
that is ``simple'' in a Hilbert-space sense (small factorization norm), and
that respects a group symmetry, should be expressible with few signed
equality kernels that respect the same symmetry. Proving exactly this
statement, with a clean quantitative bound, is the point of this paper.

The proof builds on the decomposition framework developed by Goh--Hatami
\cite{GH} and Beke--Goh--Hatami--Jaffe--Naylor \cite{BGHJN}. In particular,
the near-constancy argument and the local norm-decreasing split adapt key
ingredients from those works, while the subsequent binary recursion follows
the architecture of \cite{BGHJN}. The new ingredients are the martingale
estimate in Theorem~\ref{thm:martingale}, which improves the quantitative
exponent, and the equivariant orbitwise implementation, which yields the
non-abelian consequence.

\subsection{Equality kernels and the master theorem}

We now make the notion of equality kernel precise, since everything else in
the paper is built out of it.

\begin{definition}[Equality kernel]
Let $X$ and $Y$ be sets. A kernel $B:X\times Y \to \{0,1\}$ is an
\emph{equality kernel} if there exist a set $\Lambda$ (the label set) and maps
$$
\phi : X \to \Lambda \cup \{\bot\}, \qquad \psi : Y \to \Lambda \cup \{\bot\}
$$
such that for all $x\in X$ and $y \in Y$,
$$
B(x,y) = 1 \iff \phi(x) = \psi(y) \in \Lambda .
$$
If a group $\Gamma$ acts on $X$ and on $Y$, we say $B$ is \emph{$\Gamma$-equivariant}
if $B(gx,gy) = B(x,y)$ for every $g\in\Gamma$.
\end{definition}

The role of $\bot$ is simply bookkeeping: it lets a row or column be
``unmatched'', contributing an all-zero row or column to $B$, without forcing
us to throw it out of $X$ or $Y$.

We also need a way of measuring how complicated a kernel is, and this is
where Hilbert space enters. The relevant notion, standard in operator theory
and in the theory of Schur multipliers, is the factorization norm.

\begin{definition}[Factorization norm]
For a real kernel $A: X\times Y \to \R$, its \emph{$\gamma_2$-norm} is
$$
\|A\|_{\gamma_2} := \inf \Big\{ \big(\sup_x \|u_x\|\big)\big(\sup_y\|v_y\|\big) : A(x,y) = \langle u_x, v_y\rangle \text{ for some Hilbert space vectors } u_x, v_y \Big\}.
$$
A factorization realizing (or nearly realizing) this infimum with
$\sup_x \|u_x\| \le 1$ and $\sup_y \|v_y\| \le \gamma$ is called a
\emph{$\gamma$-factorization}.
\end{definition}

This is a genuinely flexible notion of complexity: it does not care how large
$X$ and $Y$ are, only about how many dimensions of ``genuine variation'' are
needed to represent $A$ as a Gram-type matrix. In particular it makes sense to
ask for a $\gamma_2$-factorization that is compatible with a group action, and
this compatibility is what turns the analytic statement into a structural
one.

\begin{definition}[Equivariant factorization]
Let a group $\Gamma$ act on sets $X$ and $Y$. An \emph{equivariant
$\gamma$-factorization} of a kernel $A: X\times Y \to \R$ consists of a real
Hilbert space $\mathcal H$, an orthogonal representation
$R:\Gamma \to O(\mathcal H)$, and vectors $u_x, v_y \in \mathcal H$ such that
$$
A(x,y) = \langle u_x, v_y\rangle, \qquad \|u_x\| \le 1, \qquad \|v_y\| \le \gamma,
$$
and the vectors transform correctly under the action:
$$
u_{gx} = R_g u_x, \qquad v_{gy} = R_g v_y \qquad \text{for all } g \in \Gamma.
$$
\end{definition}

The requirement $u_{gx}=R_gu_x$ is the natural thing to ask for: it says that
moving a row by the group action corresponds to applying a single fixed
orthogonal transformation to its associated vector, uniformly across the
whole space, rather than something that could vary erratically from orbit to
orbit.

We can now state the theorem around which the whole paper is organized.

\begin{theorem}[Equivariant integral-kernel decomposition]
\label{thm:master}
There is an absolute constant $C>0$ with the following property. Let a group
$\Gamma$ act \emph{freely} on sets $X$ and $Y$, and suppose that the column
orbit space $Y/\Gamma$ is finite. Let $K: X\times Y \to \Z$ be
diagonally $\Gamma$-invariant, meaning $K(gx,gy)=K(x,y)$ for all $g$, and
suppose $K$ has an equivariant $\gamma$-factorization. Then
$$
K = \sum_{i=1}^L \sigma_i B_i, \qquad L \le 2^{C\gamma^4},
$$
where each $\sigma_i \in \{-1,1\}$ and each $B_i$ is a $\Gamma$-equivariant
equality kernel.
\end{theorem}

The roles of the freeness and finiteness hypotheses are as follows. Freeness on the row side is used in the terminal equality
decomposition, where choices made on row-orbit representatives are transported
equivariantly. Freeness on the column side is used in the splitting step,
where a split constructed at one column is transported uniquely around its
orbit. Finiteness is needed only for the column orbit space $Y/\Gamma$, so
that the orbit-processing loop terminates after finitely many steps. No
finiteness assumption on $X/\Gamma$ is used.

The two motivating applications are obtained by specializing the group
action, as we now explain.

\subsection{The trivial action: matrices}

If $\Gamma$ is the trivial group, then freeness is automatic and the column
orbit space is just $Y$, so the only finiteness hypothesis in the master
theorem is that $Y$ be finite. In the finite-matrix application both $X$ and
$Y$ are of course finite, and $\Gamma$-invariance is no constraint at all.
Moreover a $\Gamma$-equivariant equality kernel is simply
an equality kernel, which by the discussion above is the same thing as a
blocky matrix. So Theorem~\ref{thm:master} immediately specializes to a
theorem purely about matrices.

\begin{corollary}[Integer matrices]
\label{cor:matrix}
There is an absolute constant $C>0$ such that every finite integer matrix $A$
with $\|A\|_{\gamma_2}\le \gamma$ can be written
$$
A = \sum_{i=1}^L \sigma_i B_i, \qquad L \le 2^{C\gamma^4},
$$
where each $B_i$ is blocky and $\sigma_i \in \{-1,1\}$. In particular this
applies to every Boolean matrix of bounded Schur-multiplier norm.
\end{corollary}

This strengthens the dimension-free theorem of Beke, Goh, Hatami, Jaffe and
Naylor \cite{BGHJN}, whose bound has $\gamma^6$ in the exponent rather than
$\gamma^4$; we explain in Section~\ref{sec:matrix} exactly where the saving
comes from.

\subsection{The regular action: groups}

At the opposite extreme, let $X = Y = G$ for a locally compact group $G$, and
let $G$ act on itself by left translation, $g\cdot x = gx$. This action is
free (if $gx=x$ then $g$ is the identity) and has a single orbit, so the
orbit-space finiteness hypothesis is satisfied trivially, however large or
infinite $G$ may be. An element of the Fourier--Stieltjes algebra $B(G)$ --
recall that this consists of the coefficient functions
$f(g)=\langle \pi(g)\xi,\eta\rangle$ of strongly continuous unitary
representations $\pi$ of $G$, with $\|f\|_{B(G)}$ the infimum of
$\|\xi\|\|\eta\|$ over such representations -- gives rise to exactly the kind
of equivariant Hilbert factorization that Theorem~\ref{thm:master} wants, once
one unwinds the definitions (we carry this out in
Section~\ref{sec:group}). The extra piece of input needed on top of the
abstract theorem is topological: because coefficient functions of continuous
representations are uniformly continuous, the subgroups that appear at the
end of the recursion are automatically open. This gives:

\begin{theorem}[Quantitative non-abelian idempotent theorem]
\label{thm:group}
There is an absolute constant $C>0$ with the following property. Let $G$ be a
locally compact group and let $f\in B(G)$ be integer-valued with
$\|f\|_{B(G)} \le \gamma$. Then there exist open subgroups $H_1,\dots,H_L \le
G$, elements $a_1,\dots,a_L \in G$, and signs $\sigma_i\in\{-1,1\}$, with
repetitions allowed, such that
$$
f = \sum_{i=1}^L \sigma_i \1_{a_iH_i}, \qquad L \le 2^{C\gamma^4}.
$$
\end{theorem}

The passage from the matrix theorem to Theorem~\ref{thm:group} is not formal.
The decomposition in Beke--Goh--Hatami--Jaffe--Naylor \cite{BGHJN} produces
arbitrary blocky matrices; even when the input has the translation-invariant
form $M_f(x,y)=f(x^{-1}y)$, the individual summands need not retain that form,
so a coset decomposition does not follow directly. The additional structural
ingredient here is an \emph{equivariant recursion}: a local split of one
column is transported around its entire group orbit using the underlying
orthogonal representation, so every child kernel remains equivariant. For the
regular action this means that every child still has the form
$A(x,y)=g(x^{-1}y)$ for a coefficient function $g$. At termination, the
rounded coefficient function is constant on cosets of its period subgroup, and
strong continuity forces that subgroup to be open. This structural
ingredient is logically distinct from the martingale estimate: equivariance is
what makes the non-abelian consequence possible, while the improved weighted
Littlestone bound is what sharpens the quantitative exponent from six to four.

In the Boolean special case for finite groups, this replaces the tower-type
bound in Sanders's quantitative non-abelian theorem \cite{SandersNA} by a
single exponential. Sanders's explicit quantitative theorem is stated for
finite groups; the present conclusion gives a uniform single-exponential
bound for arbitrary locally compact groups.

Sanders also gave a related structural description in terms of coset decision
trees: if $G$ is finite and $f:G\to\{0,1\}$ has
$\|f\|_{A(G)}\le M$, then $f$ is computed by a coset decision tree with at
most $\exp(\exp(\exp(O(M^2))))$ leaves \cite{SandersCDT}; we compare this
with our coset-decomposition statement in Section~\ref{sec:comparison}.

Restricting to a locally compact \emph{abelian} group $K$ and taking Fourier
transforms turns Theorem~\ref{thm:group} into the classical quantitative
idempotent theorem.

\begin{corollary}[Abelian quantitative idempotent theorem]
\label{cor:abelian}
Let $K$ be a locally compact abelian group and $\mu \in M(K)$ an idempotent
measure. Then
$$
\widehat\mu = \sum_{i=1}^L \sigma_i \1_{\gamma_i + \Gamma_i}, \qquad L \le 2^{O(\|\mu\|^4)},
$$
where each $\Gamma_i$ is an open subgroup of $\widehat K$.
\end{corollary}

Green and Sanders originally proved the bound
$L\le\exp\exp(C\|\mu\|^4)$ \cite{GS}. Sanders later summarized the
finite-abelian bound as $\exp(M^{4+o(1)})$ \cite{SandersA}. Thus an
$\exp(O(M^4))$ estimate removes the $M^{o(1)}$ loss in the exponent. On
$\mathbb F_2^n$, Sanders's $\exp(M^{3+o(1)})$ bound \cite{SandersF2}
remains sharper. In addition, the same argument gives a clean single-exponential bound
simultaneously for arbitrary integer matrices and for arbitrary locally
compact groups, abelian or not.

\subsection{Proof strategy}
\label{subsec:picture}

Before beginning the proof, we describe the role of its main ingredients. The
argument has several moving parts, and the following overview will be useful
when they are assembled in the recursion.

Think of $K$ as a huge integer-valued table, with a Hilbert-space vector
$v_y$ of norm at most $\gamma$ attached to each column and a unit vector $u_x$
attached to each row, so that $K(x,y)=\langle u_x,v_y\rangle$. The problem is
that we want to find columns with \emph{exactly} equal entries -- that is what
an equality kernel needs -- while all we are handed is a Euclidean picture, in
which ``closeness'' rather than ``equality'' is the natural notion. So the
argument has to do two things: it has to convert closeness into equality (by
rounding), and it has to find enough closeness to begin with (this is where
the Hilbert-space hypothesis is used).

The proof has four steps.

\begin{enumerate}[label=(\arabic*)]
\item \textbf{A dimension bound for a combinatorial game.} We first show that
if $A$ has a $\gamma$-factorization, then a certain adaptive
guessing game that one can play against the columns of $A$ -- formalized as a
weighted version of Littlestone dimension -- cannot last for more than
roughly $\gamma^2$ rounds. This is Theorem~\ref{thm:martingale}, and it is
proved by a short martingale computation: along a random play of the game,
the ``mistakes'' made by an adversary correspond to increments of a
random walk in Hilbert space, and predictability of the game (the row played
at each round cannot depend on the future) is exactly the condition that
makes the increments orthogonal in expectation, so that the total squared
displacement after $d$ rounds is only $d$, not $d^2$.

\item \textbf{From a dimension bound to near-constant columns.} A finite
bound on this combinatorial dimension is then turned, by an induction and
pigeonhole argument in the spirit of Sauer--Shelah type proofs, into the
existence of a \emph{large} set of columns on which every row is nearly
constant. This is
Proposition~\ref{prop:near} and Corollary~\ref{cor:near}.

\item \textbf{Averaging produces a norm-decreasing split.} Given such a large
near-constant family of columns, we average their Hilbert-space vectors. The
common nonzero rounded value in one row forces the average vector to have
nontrivial norm, and Lemma~\ref{lem:average} then makes one difference
$v_{y_0}-\widehat v$ smaller than the ambient norm bound. Independently,
pairwise separation of the distinct rounded columns and the bias--variance
identity force $\widehat v$ itself to be smaller than the ambient bound.
Thus $v_{y_0}=(v_{y_0}-\widehat v)+\widehat v$ is a split into two pieces of
strictly smaller norm, both still nearly integer-valued against every row. This is
Lemma~\ref{lem:local-split}, the real combinatorial heart of the paper.

\item \textbf{Recursing, and stopping.} We keep applying step (3), which
strictly decreases a squared-norm potential each time, so it can only happen
boundedly many times (at most $O(\gamma^2)$, since each split costs at least
$1/16$). Eventually either the potential is exhausted, or the reduced row mass
$\Dred$ falls below the splitting threshold. At that point we stop and use a
greedy equivariant subtraction argument: each subtraction removes one unit of
reduced row mass from every nonzero row, so an integer kernel of reduced row
mass $D$ is a signed sum of at most $2D$ equality kernels
(Lemma~\ref{lem:terminal-eq}).
\end{enumerate}

To implement the equivariance in step (3), we do not split columns
independently.  A split constructed at one representative $y_0$ is transported
around the whole orbit by the representation $R$.  Freeness of the action on
$Y$ makes this unambiguous, since every point in the orbit has a unique form
$gy_0$.  For the trivial action the orbits are singletons, while for the
regular action of $G$ there is only one orbit, so the transported split is
already a global convolution split.  Section~\ref{sec:equivariant} carries out
this orbitwise construction precisely.

The rest of the paper carries out steps (1)--(4) in order (Sections
\ref{sec:equality}--\ref{sec:local}), assembles them into the equivariant
recursion that proves Theorem~\ref{thm:master} (Section
\ref{sec:equivariant}), and then specializes to matrices
(Section~\ref{sec:matrix}) and to groups (Section~\ref{sec:group}), before
closing with a comparison to earlier quantitative bounds in the matrix and
abelian/non-abelian settings (Sections~\ref{sec:matrix} and~\ref{sec:comparison}).

\section{Equality kernels and reduced row mass}
\label{sec:equality}

We start with the easy, purely combinatorial half of the story: understanding
equality kernels themselves, and setting up the bookkeeping device --
\emph{reduced row mass} -- that will tell us, later, when to stop splitting
and simply read off the answer.

\subsection{Equality kernels are exactly blocky matrices}

We promised in the introduction that equality kernels and blocky matrices are
the same notion; we now confirm this formally, since it is used at once
in Corollary~\ref{cor:matrix}.

\begin{proposition}
\label{prop:block-eq}
A Boolean matrix is blocky if and only if it is an equality kernel.
\end{proposition}

\begin{proof}
Suppose $B$ is blocky, so that its support decomposes as a disjoint union
$\operatorname{supp}(B) = \bigcup_\lambda X_\lambda \times Y_\lambda$, with the
$X_\lambda$ pairwise disjoint and the $Y_\lambda$ pairwise disjoint. Label
every element of $X_\lambda$ and of $Y_\lambda$ by $\lambda$, and label every
row or column not appearing in any block by $\bot$. Then $B(x,y)=1$ exactly
when $x$ and $y$ received the same non-$\bot$ label, which is the definition
of an equality kernel.

Conversely, if $B(x,y)=1 \iff \phi(x)=\psi(y)\in\Lambda$, then for each
$\lambda\in\Lambda$ let $X_\lambda = \phi^{-1}(\lambda)$ and $Y_\lambda =
\psi^{-1}(\lambda)$; these are the desired blocks.
\end{proof}

This is a triviality, but it is the triviality that makes the whole paper
possible: it says that proving a decomposition theorem with equality kernels
as atoms, in the case of the trivial group action, \emph{is} the theorem
about blocky matrices.

\subsection{Reduced row mass}

Now suppose a group $\Gamma$ acts on $X$ and $Y$, and $K:X\times Y\to \Z$ is
diagonally invariant. We are going to want to measure, at any stage of the
argument, how far $K$ is from being expressible with a bounded number of
equality kernels. The right measure turns out not to be the naive row sum
$\sum_y |K(x,y)|$, which can be enormous simply because many columns happen to
agree with one another; instead we should count each \emph{distinct} column
only once.

\begin{definition}[Column equivalence and reduced row mass]
For an integer kernel $K:X\times Y \to \Z$, define an equivalence relation on
$Y$ by
$$
y \sim_K y' \iff K(x,y) = K(x,y') \text{ for every } x \in X,
$$
and write $\mathcal C(K) = Y/{\sim_K}$ for the resulting set of \emph{distinct
columns}. Because $K$ is $\Gamma$-invariant, the relation $\sim_K$ is
$\Gamma$-invariant too, so $\Gamma$ acts on $\mathcal C(K)$ by $g[y] = [gy]$,
and for a class $C \in \mathcal C(K)$ the value $K(x,C)$ is unambiguous.
Define the \emph{reduced row mass} of $K$ to be
$$
\Dred(K) := \sup_{x\in X} \sum_{C \in \mathcal C(K)} |K(x,C)|,
$$
interpreting a sum of nonnegative terms over a possibly infinite index set as
the supremum of its finite subsums. For a real kernel $A$ that is close to an
integer kernel (made precise in Section~\ref{sec:hilbert-mistake}), we set
$\Dred(A) := \Dred(A_\Z)$, where $A_\Z$ is its entrywise rounding.
\end{definition}

For a finite matrix, $\Dred(K)$ is simply the largest row sum you get after
you have thrown away duplicate columns -- exactly the quantity relevant to the
following observation, which says that once $\Dred$ is finite, we are already
done, and moreover done cheaply.

\begin{lemma}[Terminal equality decomposition]
\label{lem:terminal-eq}
Let $\Gamma$ act freely on $X$, and let
$K:X\times Y\to\Z$ be $\Gamma$-invariant with $\Dred(K) = D < \infty$. Then
$$
K = \sum_{i=1}^L \sigma_i B_i, \qquad L \le 2D,
$$
with each $B_i$ a $\Gamma$-equivariant equality kernel.
\end{lemma}

\begin{proof}
We first treat the case $K \ge 0$; the general case follows by applying this to the positive and negative
parts $K^+$ and $K^-$ separately. Passing from $K$ to either $K^+$ or $K^-$
can only merge column classes, so each has reduced row mass at most $D$;
moreover $K=K^+-K^-$.

Pick one representative row $x$ from each orbit of $X$. For each such
representative whose row is not identically zero, choose \emph{some} column
class $C_x \in \mathcal C(K)$ with $K(x,C_x) > 0$: such a class exists because
the row is nonzero. We now extend this single choice, made once per orbit, to
the whole orbit equivariantly, by declaring $C_{gx} := gC_x$. Freeness of the
action on $X$ guarantees that every point of $X$ is $gx$ for a unique $g$
(once we have fixed the representative $x$ of its orbit), so this is a
genuine, well-defined function of all of $X$, not just of the representatives.

Now define
$$
B(x,y) = 1 \iff [y] = C_x
$$
on rows with nonzero $K$, and $B \equiv 0$ on rows belonging to an orbit
where $K$ vanishes identically. This $B$ is manifestly an equality kernel: the
labelling map on rows sends $x$ to $C_x$ (or to $\bot$ on zero rows), and the
labelling map on columns sends $y$ to its class $[y]$. It is
$\Gamma$-equivariant by construction, since $C_{gx}=gC_x$ was arranged
precisely so that the definition of $B$ commutes with the action.

The point of subtracting $B$ from $K$ is that, on every row $x$ where the
row was nonzero, the value on the selected old column class $C_x$ decreases
by exactly $1$, and it stays nonnegative. Moreover $B$ is constant on every
old column class of $K$, so passing from $K$ to $K-B$ can only merge old
column classes; it cannot split one. Hence, for each row, the reduced row
mass of $K-B$ is at most the sum of the absolute values on the old classes,
and that sum is exactly one smaller than before on every nonzero row.
Repeating this operation can therefore happen at most $D$ times before every
row is zero, i.e.\ before $K$ itself is exhausted. This proves the lemma with $L\le D$ in the case
$K\ge0$, and hence $L\le 2D$ after combining $K^+$ and $K^-$.
\end{proof}

We will meet $\Dred$ again as the quantity that decides, at every node of the
recursion in Section~\ref{sec:equivariant}, whether we should stop (invoking
Lemma~\ref{lem:terminal-eq}) or split further.

\section{Hilbert factorizations and the weighted mistake tree}
\label{sec:hilbert-mistake}

We now set up the analytic side of the argument: what it means for a kernel
to be ``almost'' integer-valued, and the combinatorial game -- the weighted
mistake tree -- whose depth we are going to bound using the Hilbert-space
factorization.

Throughout, recall the factorization norm $\|\cdot\|_{\gamma_2}$ from the
introduction. We record for later use the elementary but useful fact that
$$
\|A\|_{\max} \le \|A\|_{\gamma_2},
$$
which follows from Cauchy--Schwarz applied to $A(x,y)=\langle u_x,v_y\rangle$.
For finite matrices, $\|\cdot\|_{\gamma_2}$ coincides with the usual
Schur-multiplier norm.

Rounding will be used repeatedly, so we fix conventions. For $t \in \R$,
write $t_\Z$ for a choice of nearest integer to $t$ (fixed once and for all
at half-integers, say by rounding down). A real kernel $A$ is
\emph{$\epsilon$-almost integer-valued} if $\|A - A_\Z\|_{\max} \le \epsilon$,
where $A_\Z(x,y) := (A(x,y))_\Z$.

Now to the combinatorial game. Fix $\alpha>0$; think of $\alpha$ as a
resolution below which we are not able, or not willing, to distinguish
values. The following notion is a weighted analogue of Littlestone dimension,
the parameter from online learning theory that measures how deep a binary
decision tree of ``queries'' one can shatter.

\begin{definition}[Weighted mistake tree and its dimension]
A \emph{weighted mistake tree of depth $d$ over $X$} is a complete binary
tree of depth $d$ whose internal nodes $\nu$ are labelled by pairs
$(x(\nu), w(\nu)) \in X \times \R$: a row to query, and a threshold. We think
of the tree as a strategy for an adaptive guesser who, at each internal node,
asks ``is the entry in row $x(\nu)$ large or small relative to $w(\nu)$?''
and moves left or right accordingly.

A matrix $A \in \R^{X\times Y}$ \emph{$\alpha$-shatters} the tree if every
root-to-leaf path can be ``realized'' by some column: that is, for every path
there is a $y \in Y$ such that at each internal node $\nu$ on the path,
$$
A(x(\nu), y) \ge w(\nu) + \alpha/2 \quad \text{if the path turns left at } \nu,
$$
$$
A(x(\nu), y) \le w(\nu) - \alpha/2 \quad \text{if the path turns right at } \nu.
$$
The largest depth of a tree that $A$ $\alpha$-shatters is the \emph{weighted
Littlestone dimension} $\Ldim_\alpha(A)$.
\end{definition}

The intuition is that if $\Ldim_\alpha(A)$ is large, then $A$ contains a rich
combinatorial structure: for every one of the exponentially many strategies
the adversary (choosing left/right moves) might follow, there really is a
column of $A$ that is consistent with every single query along that
particular path, always by a margin of at least $\alpha/2$. This is a strong
demand, and our first real theorem says that a bounded factorization norm
makes it impossible to sustain for very long.

One more piece of notation before we begin: we will run this game along a
\emph{random} path, generated by independent uniform signs
$\varepsilon_1,\dots,\varepsilon_d \in \{-1,1\}$, going left when
$\varepsilon_t=1$ and right when $\varepsilon_t=-1$. Write $\mathcal F_t =
\sigma(\varepsilon_1,\dots,\varepsilon_t)$ for the natural filtration. The
crucial structural feature of the tree is that the row $x_t$ and threshold
$w_t$ queried at level $t$, having been determined by the tree structure and
by $\varepsilon_1,\dots,\varepsilon_{t-1}$ alone, are $\mathcal
F_{t-1}$-measurable: in probabilists' language, they are \emph{predictable}.
Predictability is the key property used in the martingale estimate below.

\section{The martingale bound on weighted Littlestone dimension}

\begin{theorem}[Martingale weighted Littlestone bound]
\label{thm:martingale}
For every real matrix $A$ and every $\alpha>0$,
$$
\Ldim_\alpha(A) \le \frac{4\|A\|_{\gamma_2}^2}{\alpha^2}.
$$
\end{theorem}

We first describe the idea, since the short proof is easiest to read with the
probabilistic picture in mind. Fix a
$\gamma$-factorization of $A$ and suppose $A$ $\alpha$-shatters a tree of
depth $d$. We are going to follow a \emph{random} root-to-leaf path. At each
level $t$ we pick up a vector $\varepsilon_t u_{x_t}$, where $x_t$ is the row
queried at that level and $\varepsilon_t=\pm1$ tells us which way the path
turned. If we sum these vectors over $t=1,\dots,d$, we get a random vector in
Hilbert space; and the shattering hypothesis is designed so that this random
vector, paired against the vector $v_y$ of a column consistent with the whole
path, is forced to be large -- roughly $\alpha d /2$. On the other hand, the
\emph{typical size} of a sum of $d$ vectors of norm at most $1$, each
multiplied by an independent random sign, is only about $\sqrt d$, not $d$ --
\emph{provided} the vectors themselves do not conspire with the signs. This
is exactly where predictability comes in: because the row $x_t$ picked at
time $t$ cannot look ahead at $\varepsilon_t$, the cross terms
$\varepsilon_s\varepsilon_t \langle u_{x_s}, u_{x_t}\rangle$ for $s<t$
average out to zero, just as in the classical fact that a martingale's
quadratic variation adds rather than multiplies. Comparing the ``forced to be
at least $\alpha d/2$'' lower bound with the ``typically about $\gamma\sqrt d$''
upper bound and solving for $d$ gives the theorem.

\begin{proof}
Fix a $\gamma$-factorization $A(x,y) = \langle u_x,v_y\rangle$,
$\|u_x\|\le1$, $\|v_y\|\le \gamma$, and suppose $A$ $\alpha$-shatters a
weighted mistake tree of depth $d$. Let $\varepsilon_1,\dots,\varepsilon_d$
be independent uniform random signs, and follow the path that turns left when
$\varepsilon_t=1$ and right when $\varepsilon_t=-1$. Write $x_t, w_t$ for the
(random, but predictable) row and threshold encountered at level $t$.

By definition of shattering, for every realization of the path there is a
column $y$ (depending on the realization) such that for every $t$,
$$
\varepsilon_t\big(\langle u_{x_t}, v_y\rangle - w_t\big) \ge \frac\alpha2 .
$$
Summing over $t=1,\dots,d$ and using linearity of the inner product in its
first argument,
$$
\Big\langle \sum_{t=1}^d \varepsilon_t u_{x_t},\, v_y \Big\rangle \ge \frac{\alpha d}{2} + \sum_{t=1}^d \varepsilon_t w_t .
$$
Since $\|v_y\| \le \gamma$, the left-hand side is at most $\gamma \big\|
\sum_t \varepsilon_t u_{x_t} \big\|$, regardless of which $y$ realizes the
path; so, taking the supremum over the (path-dependent) choice of $y$ and then
expectation over the random path, and using that $\E[\varepsilon_t w_t] = 0$
because $w_t$ is $\mathcal F_{t-1}$-measurable while $\varepsilon_t$ is
independent of $\mathcal F_{t-1}$ with mean zero,
$$
\frac{\alpha d}{2} \le \gamma\, \E\Big\| \sum_{t=1}^d \varepsilon_t u_{x_t} \Big\| .
$$
By Cauchy--Schwarz, $\E\|\cdot\| \le \big(\E\|\cdot\|^2\big)^{1/2}$, so it
remains to bound the second moment. Expand the square:
$$
\E\Big\| \sum_{t=1}^d \varepsilon_t u_{x_t} \Big\|^2
= \sum_{t=1}^d \E\|u_{x_t}\|^2 + 2\sum_{s<t} \E\big[\varepsilon_s \varepsilon_t \langle u_{x_s}, u_{x_t}\rangle\big].
$$
Fix $s<t$. Conditioning on $\mathcal F_{t-1}$, the vectors $u_{x_s}$ and
$u_{x_t}$ (which depend only on $\varepsilon_1,\dots,\varepsilon_{t-1}$, since
$x_t$ is predictable), and also $\varepsilon_s$, are all $\mathcal
F_{t-1}$-measurable, while $\varepsilon_t$ is independent of $\mathcal
F_{t-1}$ and has mean zero. Hence
$$
\E\big[\varepsilon_s\varepsilon_t \langle u_{x_s},u_{x_t}\rangle\big]
= \E\Big[ \varepsilon_s\langle u_{x_s},u_{x_t}\rangle \cdot \E[\varepsilon_t \mid \mathcal F_{t-1}] \Big] = 0.
$$
So every cross term vanishes, and, using $\|u_{x_t}\|\le1$,
$$
\E\Big\|\sum_{t=1}^d \varepsilon_t u_{x_t}\Big\|^2 = \sum_{t=1}^d \E\|u_{x_t}\|^2 \le d .
$$
Combining, $\alpha d /2 \le \gamma\sqrt d$, i.e.\ $d \le 4\gamma^2/\alpha^2$.
Finally, taking the infimum over $\gamma$-factorizations of $A$ (i.e.\ letting
$\gamma \to \|A\|_{\gamma_2}$) proves the theorem.
\end{proof}

It is instructive to compare this with the more naive route, used in earlier
work on Schur multipliers, of encoding the threshold $w_t$ directly into an
enlarged factorization (for instance by appending a coordinate). That route
gives a bound of the shape $\Ldim_\alpha(A) = O_\alpha(\gamma^4)$, quadratically worse
in $\gamma$ than Theorem~\ref{thm:martingale}. The martingale argument above
avoids this loss entirely, because it never needs to fold the thresholds
$w_t$ into the geometry at all -- they only ever interact with the mean-zero
random signs $\varepsilon_t$, and disappear upon taking expectation. This
quadratic saving is exactly what turns the final exponent of the whole paper
from six down to four; see Section~\ref{sec:matrix}.

\section{From bounded dimension to near-constant columns}

We now perform step (2) of the plan in Section~\ref{subsec:picture}: convert
the dimension bound on $\Ldim_\alpha$ into a statement that a matrix of
bounded $\gamma_2$-norm has, in every row simultaneously, a \emph{large} set
of columns on which its value barely moves. This kind of implication --
bounded combinatorial dimension forces large uniform, or near-uniform,
substructures -- is the same flavour of argument that appears in the
Sauer--Shelah lemma and its many descendants, and the proof below is a
weighted, real-valued version of the standard argument.

\begin{proposition}[Near constancy from weighted Littlestone dimension]
\label{prop:near}
Let $A \in \R^{X\times Y}$ with $Y$ finite and nonempty, let
$d=\Ldim_\alpha(A)<\infty$, let $M=\|A\|_{\max}$, and let
$\eta\in(0,1]$. Put $q := \lceil 2M/\alpha\rceil + 2$. Then there
exist $S \subseteq Y$ and $g: X \to [-M,M]$ such that
$$
|S| \ge |Y| \left(\frac\eta q\right)^d
$$
and, for every $x \in X$,
$$
\Pr_{y \in S}\big[\, |A(x,y) - g(x)| > 2\alpha \,\big] \le \eta .
$$
\end{proposition}

This is a quantitative reformulation of the weighted-Littlestone
near-constancy argument of Goh and Hatami \cite[Proposition~3.1]{GH}; we
include the proof both for completeness and to keep track of the parameters
needed below.

\begin{proof}
We argue by strong induction on $d$.

If $d=0$: for every row $x$ the values $\{A(x,y): y\in Y\}$ must have
diameter less than $\alpha$. Indeed, if two columns $y,y'$ had
$|A(x,y)-A(x,y')|\ge\alpha$, then labelling the single-node tree by $(x,w)$,
where $w$ is the midpoint of $A(x,y)$ and $A(x,y')$, would be $\alpha$-shattered
by $A$ (one of $y,y'$ realizes the left branch, the other the right), giving
$\Ldim_\alpha(A)\ge1$, a contradiction. So $S=Y$ and any $g(x)$ in this tiny
range works, in particular with the stronger conclusion that the exceptional
probability is $0$.

Now suppose $d\ge1$ and the proposition holds for every dimension strictly smaller than $d$. There are two cases.

\emph{Case A.} Suppose that for \emph{every} row $x$ there is an interval of
length $4\alpha$ containing at least $(1-\eta)|Y|$ of the values
$\{A(x,y):y\in Y\}$. Then take $g(x)$ to be the centre of such an interval
(projected into $[-M,M]$ if necessary), and take $S=Y$. Every value in the
chosen interval is at distance at most $2\alpha$ from $g(x)$, so the stated
strict-tail bound holds.

\emph{Case B.} Otherwise there is a row $x$ for which \emph{no} interval of
length $4\alpha$ captures a $(1-\eta)$-fraction of $Y$. Partition the range
$[-M,M]$ into consecutive intervals of length $\alpha$; there are at most $q$
of them, so they define at most $q$ column classes $Y_1,\dots,Y_q$ (grouping
$y$ by which interval $A(x,y)$ lands in). By pigeonhole some class, say
$Y_i$, has $|Y_i| \ge |Y|/q$.

Now, the union $Y_{i-1}\cup Y_i \cup Y_{i+1}\cup Y_{i+2}$ (indices truncated
to the valid range) lies inside a single interval of length $4\alpha$ by
construction; and by the assumption of Case B, this union cannot capture a
$(1-\eta)$-fraction of $Y$. So at least $\eta|Y|$ of the columns lie
\emph{outside} it, and by pigeonhole again, some single class $Y_j$ among
those outside classes has $|Y_j| \ge \eta|Y|/q$.

Because $Y_j$ lies outside
$Y_{i-1}\cup Y_i\cup Y_{i+1}\cup Y_{i+2}$, its interval index satisfies
$j\le i-2$ or $j\ge i+3$. Thus the value intervals supporting $Y_i$ and
$Y_j$ are separated by a gap of at least $\alpha$. Choosing a threshold $w$
at the midpoint of this gap puts one class entirely above $w+\alpha/2$ and
the other entirely below $w-\alpha/2$.

Consider the two restrictions $A|_{X\times Y_i}$ and $A|_{X\times Y_j}$. If
\emph{both} had weighted Littlestone dimension at least $d$, we could take a
shattered tree of depth $d$ for each, attach them as the two subtrees below a
new root labelled $(x,w)$, and get a tree of depth $d+1$ shattered by $A$ (a
path through the left subtree is realized by a column in $Y_i$, consistent
with the root because $Y_i$ lies above $w+\alpha/2$, and similarly on the
right) -- contradicting $\Ldim_\alpha(A) = d$. So one of the two restrictions,
say to $Y' \in \{Y_i,Y_j\}$, has $\Ldim_\alpha(A|_{X\times Y'}) \le d-1$, and
$|Y'| \ge \eta|Y|/q$.

Let
$$
 d':=\Ldim_\alpha(A|_{X\times Y'}),\qquad
 M':=\|A|_{X\times Y'}\|_{\max},\qquad
 q':=\lceil 2M'/\alpha\rceil+2.
$$
Then $d'\le d-1$ and $q'\le q$. By the strong inductive hypothesis applied
to $A|_{X\times Y'}$ with the same $\eta$, there are $S\subseteq Y'$ and
$g$ such that
$$
 |S|\ge |Y'|\left(\frac{\eta}{q'}\right)^{d'}
 \ge |Y'|\left(\frac{\eta}{q}\right)^{d-1}
 \ge |Y|\left(\frac{\eta}{q}\right)^d.
$$
Here we used $0<\eta/q\le1$, together with $d'\le d-1$, in the second
inequality. The required near-constancy conclusion is inherited from the
inductive hypothesis. This completes the induction.
\end{proof}

Specializing the parameters gives the version we will actually use.

\begin{corollary}[Large near-constant set]
\label{cor:near}
There is an absolute constant $C$ such that if $A\in\R^{X\times Y}$ with
$Y$ finite and nonempty, $\|A\|_{\gamma_2}\le\gamma$, $\|A\|_{\max}\le M$, and
$\eta \in (0,1]$, then there exist $S\subseteq Y$
and $g:X\to[-M,M]$ with
$$
|S| \ge |Y|\left(\frac{\eta}{C(1+M)}\right)^{C\gamma^2}
$$
and $\Pr_{y\in S}[\,|A(x,y)-g(x)|>1/4\,] \le \eta$ for every $x\in X$.
\end{corollary}

\begin{proof}
Apply Proposition~\ref{prop:near} with $\alpha=1/8$, so that ``$2\alpha$''
there becomes $1/4$, and bound $d=\Ldim_{1/8}(A) \le 256\gamma^2$ using
Theorem~\ref{thm:martingale}; also bound $q = \lceil16M\rceil+2 \le
C(1+M)$. Substituting these into Proposition~\ref{prop:near} gives the stated
exponent, after absorbing constants into $C$.
\end{proof}

\section{The local norm-decreasing split}
\label{sec:local}

We come now to step (3): the combinatorial heart of the paper. Everything so
far has produced, from a bound on $\|A\|_{\gamma_2}$, a large set of columns
$S$ on which every row is nearly constant. We now show how to exploit this to
peel a genuine chunk of squared norm off a single column vector, while
keeping every row of the two resulting pieces close to an integer. Fix,
throughout this section, absolute constants $C_0,C_1\ge1$, chosen large
enough for the proof to work, and set
$$
T(\gamma,\epsilon) := 2\left(\frac{C_0\gamma^2}{\epsilon}\right)^{C_1\gamma^2}.
$$
This threshold function will reappear constantly: it is the size below which
the reduced row mass $\Dred$ is small enough that we should stop splitting and
instead invoke Lemma~\ref{lem:terminal-eq}.

We first isolate an elementary Hilbert-space fact, essentially just the
parallelogram law in disguise, which is what forces a genuine norm decrease
whenever we average several vectors that are not too spread out.

\begin{lemma}[A vector close to its average]
\label{lem:average}
Let $v_1,\dots,v_r$ be vectors in a Hilbert space with $\|v_i\|\le\gamma$ for
all $i$, and let $\widehat v = r^{-1}\sum_i v_i$ be their average. Then some
$i$ satisfies
$$
\|v_i - \widehat v\|^2 \le \|v_i\|^2 - \frac{\|\widehat v\|^2}{2} .
$$
\end{lemma}

\begin{proof}
The identity
$$
\frac1r\sum_{i=1}^r \|v_i - \widehat v\|^2 = \frac1r\sum_{i=1}^r\|v_i\|^2 - \|\widehat v\|^2
$$
is the usual bias-variance decomposition (expand the left side and use that
$\sum_i (v_i-\widehat v) = 0$). If the desired inequality failed for
\emph{every} $i$ -- that is, if $\|v_i-\widehat v\|^2 > \|v_i\|^2 -
\tfrac12\|\widehat v\|^2$ for every $i$ -- then averaging this strict
inequality over $i$ would give
$$
\frac1r\sum_i \|v_i-\widehat v\|^2 > \frac1r\sum_i \|v_i\|^2 - \frac12\|\widehat v\|^2,
$$
which contradicts the identity above (it would force $-\|\widehat v\|^2 > -\tfrac12\|\widehat v\|^2$,
i.e.\ $\|\widehat v\|^2 < 0$). So the inequality must hold for some $i$.
\end{proof}

The point of this innocuous lemma, in context, will be to guarantee that
either the average vector $\widehat v$ is small, or else some individual
vector $v_{y_0}$ in our near-constant family sits substantially closer to
$\widehat v$ than its own norm would naively suggest -- either way, we will
be able to write $v_{y_0} = v' + v''$ with both pieces of norm strictly less
than $\gamma$.

We can now state and prove the key local lemma.

The splitting mechanism below adapts
\cite[Lemma~2.4]{BGHJN}. The sharper threshold comes from replacing the
$O(\gamma^4)$ weighted-Littlestone estimate used there by the
$O(\gamma^2)$ martingale estimate of Theorem~\ref{thm:martingale}.

\begin{lemma}[Improved norm-decreasing column split]
\label{lem:local-split}
Let
$\gamma > 1/2$ and $0<\epsilon<1/6$, and let $A \in \R^{X\times Y}$ be a
finite $\epsilon$-almost integer-valued matrix equipped with a fixed
$\gamma$-factorization $A(x,y)=\langle u_x,v_y\rangle$. Suppose that $A_\Z$
has no two equal columns, and that
$$
\max_x \sum_y |A_\Z(x,y)| \ge T(\gamma,\epsilon).
$$
Then there exist $y_0\in Y$ and vectors $v',v''$ with
$$
v_{y_0} = v'+v'', \qquad \max\{\|v'\|^2,\|v''\|^2\} \le \gamma^2 - \frac1{16},
$$
and, for every $x\in X$,
$$
\dist(\langle u_x,v'\rangle,\Z) \le 3\epsilon, \qquad \dist(\langle u_x,v''\rangle,\Z)\le3\epsilon .
$$
\end{lemma}

\begin{proof}
The strategy is: find a row $x_0$ with enormous reduced mass; extract, from
that row, a huge class $S$ of columns that all take the same rounded value
$b\ne0$ in row $x_0$; shrink $S$, using Corollary~\ref{cor:near}, to a subset
$S'$ that is near-constant \emph{in every row simultaneously}; then average
the vectors $v_y$ over $S'$, and apply Lemma~\ref{lem:average}. Because the rounded columns indexed by $S'$ are pairwise distinct
(as full columns, although they agree at $x_0$), their representing vectors
are separated in Hilbert space; this average cannot be too large, forcing
the desired norm decrease; and because $S'$ is near-constant on every row,
the average $\widehat v$ is nearly as good as any single column of $S'$ at
reproducing near-integer values in every row -- which is what lets us split
$v_{y_0}$ into two pieces that are \emph{both} still almost-integer-valued.
We now carry this out.

Write $M=\|A\|_{\max}$ and $D = \max_x\sum_y|A_\Z(x,y)|\ (\ge T(\gamma,\epsilon)>0)$.
Since $D>0$, some rounded entry of $A$ is nonzero, and an entry rounding to a
nonzero integer must itself have absolute value at least $1-\epsilon$
(else it would round to $0$); so $M \ge 1-\epsilon > 1/2$. On the other hand
$M \le \|A\|_{\gamma_2} \le \gamma$ by the factorization.

Fix a row $x_0$ realizing $D$. Every nonzero rounded entry of row $x_0$ has
absolute value at most $M+1/2 \le 2M$ (an entry rounding to integer $k$ lies
within $\epsilon<1/2$ of $k$, and is itself bounded by $M$, so $|k|\le
M+1/2$). Hence row $x_0$ has at least $D/(2M)$ nonzero entries. Since these
nonzero rounded entries take at most $8M$ possible integer values (they lie
in $[-2M,2M]\setminus\{0\}$, and $2M\ge1$), by pigeonhole some fixed nonzero
integer $b$ is attained on a set
$$
S = \{y : A_\Z(x_0,y) = b\}, \qquad |S| \ge \frac{D}{16M^2} .
$$

Set $\eta = \epsilon/(4M)$, and apply Corollary~\ref{cor:near} to the
restriction of $A$ to $X\times S$ (still of $\gamma_2$-norm at most $\gamma$
and max-norm at most $M$, since restricting a matrix cannot increase either
norm). We obtain $S'\subseteq S$ and $g': X\to[-M,M]$ with
$$
|S'| \ge |S|\left(\frac{\epsilon}{CM(1+M)}\right)^{C\gamma^2} \ge \frac{D}{16M^2}\left(\frac{\epsilon}{CM(1+M)}\right)^{C\gamma^2},
$$
and, for every $x$, at most an $\eta$-fraction of $S'$ has
$|A(x,y)-g'(x)|>1/4$. Since $D>0$ we already know $M\ge1-\epsilon>5/6$, and hence
$\gamma\ge5/6$. Thus
$$
 M(1+M)\le \gamma(1+\gamma)\le 3\gamma^2.
$$
Using $D\ge T(\gamma,\epsilon)$ in the preceding display, we obtain
$$
 |S'|
 \ge \frac1{8\gamma^2}
 \left(\frac{C_0\gamma^2}{\epsilon}\right)^{C_1\gamma^2}
 \left(\frac{\epsilon}{C'\gamma^2}\right)^{C\gamma^2}
$$
for an absolute $C'$. Choosing first $C_1>C$ and then $C_0$ sufficiently
large makes this quantity at least $2$ uniformly for
$\gamma\ge5/6$ and $0<\epsilon<1/6$. Hence $|S'|\ge2$.

Let $g(x)$ be a nearest integer to $g'(x)$. If
$|A(x,y)-g'(x)|\le1/4$, put $k=A_\Z(x,y)$. Since $A$ is
$\epsilon$-almost integer-valued,
$$
 |g'(x)-k|
 \le |g'(x)-A(x,y)|+|A(x,y)-k|
 \le \frac14+\epsilon
 < \frac12 .
$$
Thus $k$ is the unique nearest integer to $g'(x)$, so $k=g(x)$. Hence
$$
\Pr_{y\in S'}[A_\Z(x,y)\ne g(x)] \le \eta \qquad \text{for every } x.
$$

Now let $\widehat v := |S'|^{-1}\sum_{y\in S'} v_y$. Because every $y\in S'$ lies in $S$, we have
$A_\Z(x_0,y)=b\ne0$ for all such $y$. Hence
$$
\big|\langle u_{x_0},\widehat v\rangle-b\big|
 =\Big|\frac1{|S'|}\sum_{y\in S'}(A(x_0,y)-b)\Big|
 \le \epsilon .
$$
Since $|b|\ge1$, it follows that
$$
|\langle u_{x_0},\widehat v\rangle| \ge 1-\epsilon > \tfrac12,
$$
so $\|\widehat v\| \ge |\langle u_{x_0},\widehat v\rangle| > 1/2$ (using
$\|u_{x_0}\|\le1$ and Cauchy--Schwarz). By Lemma~\ref{lem:average} applied to
the family $\{v_y\}_{y\in S'}$, there is $y_0\in S'$ with
$$
\|v_{y_0}-\widehat v\|^2 \le \|v_{y_0}\|^2 - \tfrac12\|\widehat v\|^2 \le \gamma^2 - \tfrac18 .
$$

It remains to bound $\|\widehat v\|^2$ itself away from $\gamma^2$, and for
this we use that the columns of $S'$ are pairwise \emph{distinct} rounded
columns of a matrix with no duplicate rounded columns: any two $y,y'\in
S'\subseteq Y$ differ in $A_\Z$ in some row, so $|A(x,y)-A(x,y')|\ge1-2\epsilon>1/2$ for that row. Since
$$
 A(x,y)-A(x,y')=\langle u_x,v_y-v_{y'}\rangle
$$
and $\|u_x\|\le1$, Cauchy--Schwarz gives
$\|v_y-v_{y'}\|\ge |A(x,y)-A(x,y')|>1/2$. The pairwise form of the bias-variance identity,
$$
\frac{1}{2|S'|^2}\sum_{y,y'\in S'} \|v_y-v_{y'}\|^2 = \frac1{|S'|}\sum_{y\in S'}\|v_y\|^2 - \|\widehat v\|^2,
$$
(obtained by expanding both sides and comparing) then gives, using
$\|v_y-v_{y'}\|^2>1/4$ for the $|S'|(|S'|-1)$ ordered pairs with $y\ne y'$
and $\|v_y\|^2\le\gamma^2$,
$$
\|\widehat v\|^2 \le \gamma^2 - \frac{|S'|-1}{8|S'|} \le \gamma^2 - \frac1{16},
$$
the last step using $|S'|\ge2$.

Set $v'' := \widehat v$ and $v' := v_{y_0}-\widehat v$. We have just shown
$\|v''\|^2 \le \gamma^2-1/16$ and (from the previous paragraph)
$\|v'\|^2 \le \gamma^2-1/8 \le \gamma^2-1/16$, so both norm bounds hold.

Finally we check almost-integrality of $\langle u_x,v'\rangle$ and
$\langle u_x,v''\rangle$ for every $x\in X$. On the $(1-\eta)$-fraction of
$S'$ where $A_\Z(x,y)=g(x)$, we have $|A(x,y)-g(x)|\le\epsilon$; on the
remaining $\eta$-fraction, we only know $|A(x,y)-g(x)| \le |A(x,y)|+|g(x)|
\le M + (M+1) \le 2M+1$ crudely. Averaging over $S'$,
$$
|\langle u_x,v''\rangle - g(x)| = \Big|\frac1{|S'|}\sum_{y\in S'}\big(A(x,y)-g(x)\big)\Big| \le (1-\eta)\epsilon + \eta(2M+1) \le \epsilon + \eta(2M+1) .
$$
Recalling $\eta = \epsilon/(4M)$, we get $\eta(2M+1) \le \epsilon(2M+1)/(4M)
\le \epsilon$ (using $M\ge1/2$, so $2M+1\le 4M$ once $M\ge1/2$: indeed
$2M+1 \le 4M \iff 1\le 2M \iff M\ge1/2$, which holds). Hence $|\langle
u_x,v''\rangle - g(x)| \le 2\epsilon$, so $\dist(\langle u_x,v''\rangle,\Z)
\le 2\epsilon \le 3\epsilon$.

For $v'$: since $v_{y_0}=v'+v''$,
$$
\langle u_x,v'\rangle = A(x,y_0) - \langle u_x,v''\rangle = \big(A(x,y_0)-g(x)\big) - \big(\langle u_x,v''\rangle - g(x)\big).
$$
The first bracket is within $\epsilon$ of the integer
$A_\Z(x,y_0)-g(x)$, while the second bracket has absolute value at most
$2\epsilon$. Therefore $\langle u_x,v'\rangle$ is within $3\epsilon$ of
the integer $A_\Z(x,y_0)-g(x)$, and hence
$\dist(\langle u_x,v'\rangle,\Z) \le 3\epsilon$, completing the proof.
\end{proof}

The two halves $v',v''$ produced by this lemma are, informally, ``$v_{y_0}$
minus its near-constant average'' and ``the average itself''. Both improve
the ambient squared column-norm bound by a universal amount, $1/16$, and,
crucially, both are still nearly integer-valued when tested against
\emph{every} row $x$, not just the rows used to find the near-constant set.
This is exactly the local move that will be iterated, and transported around
group orbits, in the next section.

\section{The equivariant recursion}
\label{sec:equivariant}

We now assemble the pieces (Sections~\ref{sec:equality}--\ref{sec:local})
into a proof of Theorem~\ref{thm:master}. As foreshadowed in
Section~\ref{subsec:picture}, the only new idea needed is a way of performing
the local split of Lemma~\ref{lem:local-split} \emph{compatibly with the
group action}, and the mechanism for this is transport around a single free
orbit, which we now describe.

\subsection{Splitting one orbit}

\begin{lemma}[Splitting one orbit]
\label{lem:orbit-split}
Let $\Gamma$ act on $X$ and freely on $Y$. Let $A: X\times Y\to\R$ be
$\epsilon$-almost integer-valued, $0<\epsilon<1/32$, with an equivariant
$\gamma$-factorization. If $\Dred(A) > T(\gamma,\epsilon)$, then there is a
$\Gamma$-orbit $\mathcal O \subseteq Y$ and equivariant column families
$\{v_y^{(1)}\}_{y\in\mathcal O}$, $\{v_y^{(2)}\}_{y\in\mathcal O}$ such that
$$
v_y = v_y^{(1)}+v_y^{(2)} \quad (y\in\mathcal O), \qquad \|v_y^{(j)}\|^2 \le \gamma^2-\tfrac1{16}\ (j=1,2),
$$
and $\dist(\langle u_x,v_y^{(j)}\rangle,\Z)\le3\epsilon$ for every $x\in X$
and $y\in\mathcal O$.
\end{lemma}

\begin{proof}
Since $\Dred(A)>T(\gamma,\epsilon)>0$, some rounded entry is nonzero. Hence
$\gamma\ge\|A\|_{\max}\ge1-\epsilon>1/2$, so the hypothesis
$\gamma>1/2$ needed in Lemma~\ref{lem:local-split} is automatic. By
definition of reduced row mass there is a row $x_0$ and a \emph{finite} collection of distinct rounded-column
classes whose masses in row $x_0$ already sum to more than
$T(\gamma,\epsilon)$. Choose one representative column from each of these
classes; call the resulting finite set $Y_0\subseteq Y$. Because the classes
were chosen distinct, the rounded columns restricted to $Y_0$ are pairwise
different \emph{somewhere}; so there is a finite set of rows $X_0 \ni x_0$
that separates every pair of columns in $Y_0$ (for each of the finitely many
pairs, keep one witnessing row).

We would like to apply Lemma~\ref{lem:local-split} directly to $A$ restricted
to $X\times Y_0$, but the lemma requires the ambient row set to be finite,
whereas $X$ itself may be infinite. We get around this by a compactness
argument: for every \emph{finite} superset $X'\supseteq X_0$, apply
Lemma~\ref{lem:local-split} to the finite matrix $A|_{X'\times Y_0}$ (which
still has no duplicate rounded columns, since $X_0\subseteq X'$ already
separates them, and still has reduced row mass exceeding $T(\gamma,\epsilon)$
at $x_0$). This produces some $y_{X'}\in Y_0$ and vectors $w_{1,X'}, w_{2,X'}$
satisfying the norm-decrease and almost-integrality conclusions, but only
verified so far on the rows in $X'$.

Order the finite supersets $X'$ of $X_0$ by inclusion; this is a directed set.
Since $Y_0$ is finite, we may pass to a subnet along which the chosen column
$y_{X'}$ is eventually constant, equal to some $y_0\in Y_0$. The vectors
$w_{1,X'}$ and $w_{2,X'}$ all lie in the closed ball of radius
$\sqrt{\gamma^2-1/16}$, which is weakly compact (closed bounded balls in
Hilbert space are weakly compact); so, passing to a further subnet, we may
assume $w_{1,X'}\rightharpoonup w_1$ and $w_{2,X'}\rightharpoonup w_2$
weakly.

We now pass the conclusions of Lemma~\ref{lem:local-split} to this weak
limit. The equation $v_{y_0} = w_{1,X'}+w_{2,X'}$, valid along the net,
passes to the limit by weak continuity of addition, giving
$v_{y_0}=w_1+w_2$. The norm bound $\|w_{j,X'}\|^2 \le \gamma^2-1/16$ passes
to the limit by weak lower semicontinuity of the norm (a standard fact: if
$w_{j,X'}\rightharpoonup w_j$ then $\|w_j\|\le\liminf\|w_{j,X'}\|$), giving
$\|w_j\|^2\le\gamma^2-1/16$. Finally, for each \emph{fixed} $x\in X$, once
$X'\supseteq\{x\}$ we have from the lemma that
$\dist(\langle u_x,w_{j,X'}\rangle,\Z)\le3\epsilon$; since
$\langle u_x,\cdot\rangle$ is weakly continuous and $\dist(\cdot,\Z)$ is
continuous, this bound also passes to the limit, and it holds for
\emph{every} $x\in X$ because $x$ was an arbitrary fixed point (once in the
directed set, it stays in every larger $X'$).

We have therefore produced $y_0\in Y$ and $w_1,w_2$ with $v_{y_0}=w_1+w_2$,
$\|w_j\|^2\le\gamma^2-1/16$, and $\dist(\langle u_x,w_j\rangle,\Z)\le3\epsilon$
for every $x\in X$.

Finally, transport this single split around the whole orbit
$\mathcal O = \Gamma y_0$. Because the action on $Y$ is free, every
$y\in\mathcal O$ has a \emph{unique} representation $y=gy_0$; define
$v_{gy_0}^{(j)} := R_g w_j$. This is well-defined (no ambiguity in $g$), and
$$
v_{gy_0} = R_g v_{y_0} = R_g w_1 + R_g w_2 = v^{(1)}_{gy_0}+v^{(2)}_{gy_0},
$$
using equivariance of the original factorization, and the norms are
preserved since $R_g$ is orthogonal: $\|R_gw_j\| = \|w_j\|$. For
almost-integrality at a general point $x\in X$ and $y=gy_0\in\mathcal O$,
$$
\langle u_x, v_{gy_0}^{(j)}\rangle = \langle u_x, R_gw_j\rangle = \langle R_g^{-1}u_x, w_j\rangle = \langle u_{g^{-1}x}, w_j\rangle,
$$
using that $R_g$ is orthogonal (so $R_g^*=R_g^{-1}$) and equivariance
$u_{g^{-1}x}=R_g^{-1}u_x$. Since $g^{-1}x$ ranges over $X$ as $x$ does, this
is within $3\epsilon$ of an integer by the bound already established for all
of $X$. So the almost-integrality conclusion holds at \emph{every} point of
$X\times\mathcal O$, as claimed.
\end{proof}

The role of the compactness step is worth emphasizing. It allows us to deduce
the infinitary statement, for instance the group specialization when $G$ is
infinite, from Lemma~\ref{lem:local-split}, which is stated and proved only
for finite matrices. We only need finitely many rows at a time, and weak
compactness passes these finite conclusions to a global vector.

\subsection{One step of the recursion}

The following orbitwise decomposition is the equivariant counterpart of the
decomposition step in \cite[Lemma~2.5]{BGHJN}.

\begin{lemma}[One equivariant decomposition step]
\label{lem:eq-step}
Let $\Gamma$ act freely on $X$ and $Y$, with $Y/\Gamma$ finite. Let
$A:X\times Y\to\R$ be $\epsilon$-almost integer-valued, $0<\epsilon<1/32$,
with an equivariant $\gamma$-factorization. Then
$$
A_\Z = K_1+K_2+F,
$$
where each $K_j = (A_j)_\Z$ for an equivariant $3\epsilon$-almost
integer-valued kernel $A_j$ with an equivariant factorization satisfying
$\sup_y\|v_y^{(j)}\|^2 \le \gamma^2-1/16$, and
$$
\eqc_\Gamma(F) \le 2T(\gamma,\epsilon),
$$
where $\eqc_\Gamma(F)$ denotes the least number of signed
$\Gamma$-equivariant equality kernels needed to write $F$.
\end{lemma}

\begin{proof}
Think of every $\Gamma$-orbit of columns as either ``processed'' or
``unprocessed'', starting with all orbits unprocessed. While the restriction
of $A$ to the union of the currently unprocessed orbits still has reduced row
mass exceeding $T(\gamma,\epsilon)$, apply Lemma~\ref{lem:orbit-split} to that
restriction: it selects some unprocessed orbit and splits its column vectors
into two equivariant families, which we append to two growing families
$\{v_y^{(1)}\}$ and $\{v_y^{(2)}\}$ (defined so far only on processed
orbits), and we mark that orbit as processed.

Because $Y/\Gamma$ is finite, there are only finitely many orbits to process,
so this loop terminates after finitely many iterations. When it terminates,
either every orbit has been processed, or the restriction to the remaining
unprocessed orbits has reduced row mass at most $T(\gamma,\epsilon)$.

On the union of the processed (split) orbits, define $A_1,A_2$ using the
recorded families $v^{(1)},v^{(2)}$; on the unprocessed orbits, set both to
zero. Each $A_j$ is then $3\epsilon$-almost integer-valued (by the conclusion
of Lemma~\ref{lem:orbit-split} on split orbits, and trivially on unprocessed
orbits, where it is exactly $0$), equivariant (each split was performed
equivariantly, orbit by orbit), and satisfies the norm bound $\sup_y
\|v_y^{(j)}\|^2 \le \gamma^2-1/16$ (again by Lemma~\ref{lem:orbit-split} on
split orbits, trivially elsewhere).

On every split orbit we have the exact identity $A = A_1+A_2$ (this is how
the split was constructed); consequently, using the triangle inequality and
the almost-integrality bounds $\|A-A_\Z\|_{\max}\le\epsilon$ and
$\|A_j-(A_j)_\Z\|_{\max}\le3\epsilon$,
$$
\big| A_\Z - (A_1)_\Z - (A_2)_\Z \big|
\le |A_\Z - A| + |A - A_1 - A_2| + |A_1-(A_1)_\Z| + |A_2-(A_2)_\Z|
\le \epsilon + 0 + 3\epsilon+3\epsilon < 1
$$
pointwise on every split orbit. The left-hand side is an integer (a
difference of integers), and an integer of absolute value less than $1$ must
be $0$. So $A_\Z = K_1+K_2$ exactly on every split orbit, where
$K_j=(A_j)_\Z$.

On the unprocessed orbits, $K_1=K_2=0$ by definition, so there
$A_\Z=K_1+K_2+F$ with $F:=A_\Z$ restricted to those orbits (and $F=0$ on
split orbits). By the termination condition, $\Dred(F)\le T(\gamma,\epsilon)$
(reduced row mass can only decrease under restriction to a subset of columns,
so the mass of $F$, supported on the unprocessed orbits, is controlled by the
mass of $A$ restricted to those same orbits). $F$ is $\Gamma$-invariant
(being a restriction, to a union of orbits, of the invariant kernel $A_\Z$)
and $\Gamma$ acts freely on $X$ (inherited from the hypothesis on $X$), so
Lemma~\ref{lem:terminal-eq} applies and gives $\eqc_\Gamma(F) \le
2\Dred(F) \le 2T(\gamma,\epsilon)$.
\end{proof}

\subsection{Proof of the master theorem}

We are now ready to prove Theorem~\ref{thm:master}, by iterating
Lemma~\ref{lem:eq-step}: each application either terminates a branch (if the
residual reduced row mass was already small) or splits it into two children,
each of which has strictly smaller squared factorization norm. Since the
norm cannot decrease indefinitely, the recursion tree has bounded depth, and
we simply sum the (bounded) equality complexity picked up at every node.

\begin{proof}[Proof of Theorem~\ref{thm:master}]
If $K=0$ there is nothing to prove, so assume $K\ne0$. Since $K$ is
integer-valued and nonzero, $\|K\|_{\max}\ge1$, and therefore the initial
column-norm parameter satisfies $\gamma\ge1$. Set
$$
 \epsilon_0:=2^{-32\gamma^2-5}.
$$

It is convenient to keep track of what the recursive objects actually are.
A node carries a pair $(A,K_A)$, where $A$ has an equivariant factorization,
is $\epsilon$-almost integer-valued for some current error $\epsilon$, and
$K_A:=A_\Z$ is the integer kernel to be reconstructed. At the root we take $A=K$, so $K_A=K$. Although the actual rounding
error is $0$, we use $\epsilon_0$ as the bookkeeping error parameter in the
first application of Lemma~\ref{lem:eq-step}. If the current factorization has
$$
 \|u_x\|\le1,\qquad \|v_y\|\le\rho ,
$$
we call $\rho$ the current column-norm parameter.

Apply Lemma~\ref{lem:eq-step} to $A$. It gives
$$
 K_A=(A_1)_\Z+(A_2)_\Z+F,
$$
where $F$ is already a signed sum of at most $2T(\rho,\epsilon)$ equivariant
equality kernels, while each child with $(A_j)_\Z\ne0$ is $3\epsilon$-almost
integer-valued and has column-norm parameter $\rho_j$ satisfying
$$
 \rho_j^2\le \rho^2-\frac1{16}.
$$
Thus the children are the pairs $(A_j,(A_j)_\Z)$, not the rounded kernels
viewed as though they themselves carried the factorization.

Along any branch, every genuine split lowers the squared column-norm
parameter by at least $1/16$. Consequently a branch contains at most
$16\gamma^2
$
genuine splits. At depth $s\le16\gamma^2$, the error is at most
$$
 3^s\epsilon_0
 \le 3^{16\gamma^2}2^{-32\gamma^2-5}
 <\frac1{32},
$$
because $16\log_2 3<26$. Hence every recursive use of
Lemma~\ref{lem:eq-step} satisfies its error hypothesis. Moreover, whenever a
node has nonzero rounded kernel, some entry of $A$ is within $\epsilon<1/32$
of a nonzero integer, so
$$
 \rho\ge\|A\|_{\max}\ge1-\epsilon>\frac12.
$$
Thus the local splitting lemma is applicable whenever a split is actually
needed.

A binary tree with at most $16\gamma^2$ genuine splits on each branch has at
most $2^{16\gamma^2+1}$ nodes (a node at which no split occurs is terminal).
It remains to bound the equality complexity contributed at a node uniformly
in terms of the root parameters. Along the recursion we have
$\rho\le\gamma$ and $\epsilon\ge\epsilon_0$. Since
$C_0\gamma^2/\epsilon_0>1$,
$$
 T(\rho,\epsilon)
 \le 2\left(\frac{C_0\gamma^2}{\epsilon_0}\right)^{C_1\gamma^2}.
$$
Therefore
$$
 \log_2 T(\rho,\epsilon)
 \le O\!\left(\gamma^2\log\gamma
       +\gamma^2\log\frac1{\epsilon_0}\right)
 =O(\gamma^4).
$$
So every node contributes at most $2^{O(\gamma^4)}$ equality kernels.
Multiplying by the number of nodes gives
$$
 L\le 2^{16\gamma^2+1}\,2^{O(\gamma^4)}
   =2^{O(\gamma^4)}.
$$

Finally, the exact identity
$K_A=(A_1)_\Z+(A_2)_\Z+F$ at every internal node implies, by induction from
the leaves to the root, that the equality-kernel pieces produced throughout
the tree sum exactly to the rounded kernel at that node. At the root the
rounded kernel is $K$. This proves the theorem.
\end{proof}

\section{Matrix consequences}
\label{sec:matrix}

We record explicitly what Theorem~\ref{thm:master} says when $\Gamma$ is
trivial, one of our two motivating applications.
As explained in the introduction, for the trivial action freeness is automatic
and the column-orbit finiteness assumption reduces to finiteness of $Y$; in
the present finite-matrix application $X$ is finite as well. The trivial
action turns $\Gamma$-equivariant equality kernels into ordinary equality
kernels, which
by Proposition~\ref{prop:block-eq} are exactly blocky matrices. So
Theorem~\ref{thm:master} is, in this case, precisely
Corollary~\ref{cor:matrix}. Because $\|\cdot\|_{\gamma_2}$ coincides with the
Schur-multiplier norm on finite matrices, we may restate this for Boolean
matrices in the language of operator theory.

\begin{corollary}[Idempotent Schur multipliers]
Every finite Boolean Schur multiplier $A$ of multiplier norm at most $\gamma$
is a signed sum of at most $2^{O(\gamma^4)}$ contractive idempotent Schur
multipliers.
\end{corollary}

Here we use the standard characterization of blocky Boolean matrices as
contractive idempotent Schur multipliers; see, for example, \cite{GH,BGHJN}.

The improvement over Beke--Goh--Hatami--Jaffe--Naylor \cite{BGHJN}, whose
bound is $2^{O(\gamma^6)}$, is entirely localized in one place: the sharper
martingale bound of Theorem~\ref{thm:martingale}. Their argument, like
ours, uses a near-constancy consequence of a bound on
$\Ldim_{1/8}(A)$; the weighted-Littlestone input underlying this step is
$O(\gamma^4)$ in the earlier approach \cite{GH,BGHJN}, whereas ours is
$O(\gamma^2)$ --
precisely the saving discussed in the remark after
Theorem~\ref{thm:martingale}. Tracing this through, the logarithm of the local splitting threshold is
$$
 \log T(\gamma,\epsilon)
 =O\!\left(\Ldim_{1/8}(A)\,
          \big(\log(1+\gamma)+\log(1/\epsilon)\big)\right).
$$
The recursion requires $\log(1/\epsilon_0)=O(\gamma^2)$. Hence the old bound
$\Ldim_{1/8}(A)=O(\gamma^4)$ gives $\log T=O(\gamma^6)$, while
Theorem~\ref{thm:martingale} gives $\log T=O(\gamma^4)$. The binary recursion
tree itself contributes only an additional $O(\gamma^2)$ to the logarithm of
the number of terms, because its depth is $O(\gamma^2)$. Thus the final
exponent improves from six to four, and the saving is exactly the quadratic
improvement in the weighted Littlestone bound.

\section{The regular group action}
\label{sec:group}

We now carry out the other specialization: $X=Y=G$, with $G$ a locally
compact group acting on itself by left translation. The work here is not to
reprove the recursion -- that has already been done, once and for all, in
Section~\ref{sec:equivariant} -- but to identify correctly (i) what an
equivariant factorization looks like for a function in $B(G)$, and (ii) what
the terminal equality kernels look like for this particular action, namely
that they are exactly indicator functions of cosets, with the acting group's
role explaining why the relevant subgroups come out \emph{open}.

\subsection{Coefficient functions give equivariant factorizations}

Recall that the Fourier--Stieltjes algebra $B(G)$ of a locally compact group
$G$ consists of the coefficient functions $g\mapsto\langle\pi(g)\xi,\eta\rangle$
of strongly continuous unitary representations $\pi$ of $G$ on a complex
Hilbert space, with norm $\|f\|_{B(G)}$ the infimum of $\|\xi\|\|\eta\|$ over
all such representations of $f$ \cite{Eymard}.

\begin{lemma}[Equivariant coefficient factorization]
\label{lem:coeff}
Let $f\in B(G)$ be real-valued and let $\gamma>\|f\|_{B(G)}$. Then there are
a real Hilbert space $\mathcal H_\R$, a strongly continuous orthogonal
representation $R:G\to O(\mathcal H_\R)$, and vectors $u_x,v_y\in
\mathcal H_\R$ (for $x,y\in G$) such that
$$
 f(x^{-1}y)=\langle u_x,v_y\rangle_\R,\qquad
 \|u_x\|\le1,\quad \|v_y\|\le\gamma,
$$
and $u_{tx}=R_tu_x$, $v_{ty}=R_tv_y$ for every $t\in G$.
\end{lemma}

\begin{proof}
By the definition of the $B(G)$ norm, choose a coefficient realization
$f(g)=\langle\pi(g)\xi,\eta\rangle_\C$ with
$\|\xi\|\,\|\eta\|<\gamma$. If $f\ne0$, rescale $\xi,\eta$ so that
$\|\eta\|=1$ and $\|\xi\|<\gamma$; the case $f=0$ is trivial. Define $u_x := \pi(x)\eta$ and $v_y := \pi(y)\xi$, vectors
in the complex Hilbert space of $\pi$. Assuming, as is conventional, that
the complex inner product is linear in its first argument, we compute
$$
\langle v_y, u_x\rangle_\C = \langle \pi(y)\xi, \pi(x)\eta\rangle_\C = \langle \pi(x)^*\pi(y)\xi,\eta\rangle_\C = \langle \pi(x^{-1}y)\xi,\eta\rangle_\C = f(x^{-1}y),
$$
using unitarity of $\pi$ (so $\pi(x)^*=\pi(x)^{-1}=\pi(x^{-1})$) and the
representation property. Since $f$ is real-valued by hypothesis, the
imaginary part of $\langle v_y,u_x\rangle_\C$ vanishes identically, and the
real part of a complex inner product is exactly the real inner product
obtained by viewing the same complex Hilbert space as a real Hilbert space
$\mathcal H_\R$ (of twice the real dimension) in the usual way; write
$\langle u_x,v_y\rangle_\R$ for this real part (note we have also swapped
the order, which does not affect the real part since the real part of a
complex number equals the real part of its conjugate). So $\langle
u_x,v_y\rangle_\R = f(x^{-1}y)$.

The norms are unchanged by realification:
$\|u_x\|_\R=\|\eta\|_\C=1$ and
$\|v_y\|_\R=\|\xi\|_\C<\gamma$.
Equivariance is immediate from the representation property: $u_{tx} =
\pi(tx)\eta = \pi(t)\pi(x)\eta = \pi(t)u_x$, and $\pi(t)$, restricted to act
on the realification $\mathcal H_\R$, is exactly the orthogonal
transformation $R_t$ we want (a unitary map on a complex Hilbert space
restricts to an orthogonal map on its realification, because it preserves
the real inner product, being the real part of a preserved complex inner
product); similarly for $v_{ty}$. Strong continuity of $R$ is inherited from
strong continuity of $\pi$.
\end{proof}

This lemma allows us to apply Theorem~\ref{thm:master} to $X=Y=G$ under the
regular action: for every
$\gamma>\|f\|_{B(G)}$, an integer-valued $f\in B(G)$ gives an equivariant
$\gamma$-factorization of the kernel $K(x,y):=f(x^{-1}y)$, which is
manifestly diagonally invariant under left translation (replacing $x,y$ by $tx,ty$ leaves $x^{-1}y$ unchanged) and, since
there is only one column orbit on $Y=G$, satisfies the orbit-finiteness
hypothesis trivially.

\subsection{Why the terminal subgroups are open}

The one piece of Theorem~\ref{thm:master} that does not come for free, in
this specialization, is that the subgroups $H_i$ appearing in the final
answer should be \emph{open}: this is a genuinely topological fact, not
visible in the purely algebraic/Hilbert-space argument of
Section~\ref{sec:equivariant}, and it is exactly where the hypothesis that
$f\in B(G)$ (rather than an arbitrary function with a Hilbert factorization)
is used.

For an integer-valued function $k:G\to\Z$, define its \emph{right-period
subgroup}
$$
\PerR(k) := \{h\in G : k(zh)=k(z) \text{ for every } z\in G\}
$$
(a routine check confirms this is indeed a subgroup: it contains the
identity, and is closed under products and inverses).

\begin{lemma}[Period subgroups are open]
\label{lem:period-open}
Let $g:G\to\R$ be a coefficient of a strongly continuous orthogonal
representation (equivalently, after complexification, of a strongly
continuous unitary representation), and let $0<\delta<1/2$ be such that $g$ is
$\delta$-almost integer-valued. Then $H := \PerR(g_\Z)$ is open.
\end{lemma}

\begin{proof}
Coefficient functions of strongly continuous unitary representations are
(right) uniformly continuous: given $\xi,\eta$ with $g(x)=\langle\pi(x)\xi,\eta\rangle$,
strong continuity of $\pi$ at the identity gives a neighbourhood $U$ of $e$
such that $\|\pi(h)\xi-\xi\|<\varepsilon'$ for $h\in U$, and then
$$
|g(xh)-g(x)| = |\langle \pi(x)(\pi(h)\xi-\xi),\eta\rangle| \le \|\pi(h)\xi-\xi\|\,\|\eta\| < \varepsilon'\|\eta\|
$$
uniformly in $x$ (using that $\pi(x)$ is unitary, hence norm-preserving).
Choosing $\varepsilon'$ small enough that $\varepsilon'\|\eta\| < 1-2\delta$
(possible since $\delta<1/2$), we obtain a neighbourhood $U$ of the identity
with
$$
\sup_x |g(xh)-g(x)| < 1-2\delta \qquad \text{for all } h\in U.
$$
Now suppose $h\in U$ but $g_\Z(xh)\ne g_\Z(x)$ for some $x$: then $g_\Z(xh)$
and $g_\Z(x)$ are distinct integers, so differ by at least $1$, while
$g(xh)$ and $g(x)$ each lie within $\delta$ of their respective rounded
values; the triangle inequality then forces
$$
|g(xh)-g(x)| \ge 1 - 2\delta,
$$
contradicting the bound above. So no such $x$ exists, i.e.\ $U\subseteq H$;
since $U$ is a neighbourhood of the identity contained in the subgroup $H$,
and subgroups containing a neighbourhood of the identity are open (translate
the neighbourhood to cover every point of $H$), $H$ is open.
\end{proof}

\subsection{Duplicate columns are exactly cosets}

For $k:G\to\Z$, consider the associated convolution kernel
$M_k(x,y):=k(x^{-1}y)$.

\begin{lemma}[Duplicate columns are cosets]
\label{lem:dup-coset}
Let $H=\PerR(k)$. Two columns $y,y'\in G$ of $M_k$ are equal (as functions
of $x$) if and only if $yH=y'H$.
\end{lemma}

\begin{proof}
Write $y'=yh$. The columns indexed by $y$ and $y'$ agree, i.e.\
$k(x^{-1}y)=k(x^{-1}y')$ for every $x$, if and only if, substituting
$z=x^{-1}y$ (which ranges over all of $G$ as $x$ does), $k(z)=k(zh)$ for every
$z\in G$ -- which is exactly the definition of $h\in H=\PerR(k)$. So the
columns agree iff $h\in H$, i.e.\ iff $yH=y'H$.
\end{proof}

Consequently, in the regular action, the equivalence classes $\mathcal
C(M_k)$ of Section~\ref{sec:equality} are exactly the left cosets
$\{yH:y\in G\}$ (the subgroup is written on the right), and the reduced
row mass becomes
$$
\Dred(M_k) = \sum_{aH\in G/H} |k(a)| .
$$
When this is finite, $k$ is trivially a finite integer combination of coset
indicators, $k=\sum_{aH} k(a)\1_{aH}$, and expanding each integer coefficient
into a sum of signs recovers a signed coset decomposition, exactly as
Lemma~\ref{lem:terminal-eq} promises in this special case.

It remains to identify the equivariant equality kernels for the regular
action with cosets in general (not just in the terminal, small-$\Dred$
case), confirming the picture sketched in the introduction.

\begin{proposition}
A nonzero translation-invariant Boolean kernel $B(x,y)=\1_S(x^{-1}y)$, for
$S\subseteq G$, is an equality kernel if and only if $S$ is a left coset of
some subgroup of $G$.
\end{proposition}

\begin{proof}
($\Leftarrow$) If $S=aH$ for a subgroup $H\le G$, then
$$
\1_{aH}(x^{-1}y) = 1 \iff x^{-1}y\in aH \iff (xa)H = yH,
$$
so $B$ is the equality kernel with labels in $G/H$, given by $\phi(x)=(xa)H$
and $\psi(y)=yH$.

($\Rightarrow$) Suppose $B(x,y)=\1_S(x^{-1}y)$ is a nonzero equality kernel.
By Proposition~\ref{prop:block-eq}, $B$ is blocky: writing the row associated
to $x$ as $\{y : x^{-1}y\in S\} = xS$, blockiness says that the family of
sets $\{xS : x\in G\}$ is pairwise either identical or disjoint. Fix any
$a\in S$ (possible since $S\ne\emptyset$, as $B\ne0$), and set
$$
H := \{g\in G : gS = S\}.
$$
This is a subgroup: $eS=S$; if $gS=S$ and $g'S=S$ then $(gg')S=g(g'S)=gS=S$;
and if $gS=S$ then applying $g^{-1}$ to both sides gives $S=g^{-1}S$.

Now take any $s\in S$. The translate $(sa^{-1})S$ contains $(sa^{-1})a = s$,
and $S$ also contains $s$; so $(sa^{-1})S$ and $S$ are two members of our
pairwise-disjoint-or-equal family that share the point $s$, hence they are
equal: $(sa^{-1})S = S$. This says $sa^{-1}\in H$, i.e.\ $s\in Ha$. Since $s$
was an arbitrary element of $S$, this shows $S\subseteq Ha$.

Conversely, if $h\in H$, then $a\in S$ and $hS=S$ imply $ha\in S$.
Thus $Ha\subseteq S$. Combining the two inclusions, $S=Ha$, which is a right coset of
$H$; writing $Ha = a(a^{-1}Ha)$ exhibits it as a left coset of the conjugate
subgroup $a^{-1}Ha$, as claimed.
\end{proof}

\subsection{Proof of the non-abelian idempotent theorem}

\begin{proof}[Proof of Theorem~\ref{thm:group}]
The case $f=0$ is trivial. Put $M=\|f\|_{B(G)}$; then $M\ge1$ because $f$ is
nonzero and integer-valued. Choose any $\gamma$ with
$$
 M<\gamma\le 2M.
$$
Lemma~\ref{lem:coeff} gives an equivariant $\gamma$-factorization of
$K(x,y)=f(x^{-1}y)$ under the regular left action of $G$ on itself. This
action is free and has one orbit on each of $X=G$ and $Y=G$, so
the recursion in the proof of Theorem~\ref{thm:master} applies.

We verify that every equality piece produced by the recursion can be chosen to
be the indicator of an \emph{open} coset. The row vectors retain the form
$u_x=R_xu_e$ throughout. If a split at a node is first constructed at a
column $y_0$, Lemma~\ref{lem:orbit-split} transports a child vector as
$v^{(j)}_{gy_0}=R_gw_j$.
Writing $y=gy_0$ gives
$$
 v^{(j)}_y=R_{yy_0^{-1}}w_j
          =R_y\bigl(R_{y_0^{-1}}w_j\bigr),
$$
so on every child the column vectors again have the form $v_y=R_yw$ for a
fixed vector $w$. Hence every real kernel occurring at a node has the form
$$
 A(x,y)=\langle R_xu_e,R_yw\rangle
       =\langle u_e,R_{x^{-1}y}w\rangle
       =g(x^{-1}y),
$$
where $g$ is a coefficient of the same strongly continuous orthogonal
representation (equivalently, after complexification, of a strongly
continuous unitary representation), and $A$ is $\delta$-almost
integer-valued with $\delta<1/32$.

There is one further point specific to the regular action. At every node there
is exactly one column orbit. Therefore, if $\Dred(A)>T(\rho,\delta)$, the
orbit-processing loop in Lemma~\ref{lem:eq-step} processes that unique orbit
and leaves \emph{no} residual term $F$. If instead
$\Dred(A)\le T(\rho,\delta)$, no split occurs and the entire rounded kernel
is terminal. Thus every nonzero residual contribution in the abstract
recursion occurs here only at a terminal node.

At such a terminal node put $k=g_\Z$ and $H=\PerR(k)$. Lemma~\ref{lem:period-open}
shows that $H$ is open. By Lemma~\ref{lem:dup-coset}, the distinct columns of
the rounded kernel are indexed by the cosets $aH$, and
$$
 \Dred(M_k)=\sum_{aH\in G/H}|k(a)|.
$$
This quantity is at most the relevant threshold, so only finitely many cosets
have nonzero value and
$$
 k=\sum_{aH\in G/H} k(a)\,\1_{aH}.
$$
Expanding each nonzero integer coefficient into $|k(a)|$ copies with the
appropriate sign gives an exact signed decomposition into indicators of
cosets of the open subgroup $H$. Its number of terms is
$\Dred(M_k)$, so it is no larger than the $2T(\rho,\delta)$ terminal cost
used in the abstract master-theorem count. Consequently the same recursion
bound gives at most $2^{C\gamma^4}$ open-coset terms. Since
$\gamma\le2M$, this is at most $2^{C' M^4}$ after changing the absolute
constant. Hence
$$
 f=\sum_{i=1}^L \sigma_i\1_{a_iH_i},
 \qquad L\le2^{C'M^4},
$$
with every $H_i$ open.
\end{proof}

For finite groups $G$, the Fourier algebra $A(G)$ and Fourier--Stieltjes
algebra $B(G)$ coincide isometrically (every function on a finite group is a
coefficient function, and the two norms agree), giving the following clean
statement, free of any topological caveats since every subgroup of a finite
group is (trivially) open.

\begin{corollary}[Finite groups]
Let $G$ be finite and $f:G\to\Z$ satisfy $\|f\|_{A(G)}\le\gamma$. Then
$$
f = \sum_{i=1}^L \sigma_i \1_{a_iH_i}, \qquad L\le 2^{O(\gamma^4)} .
$$
\end{corollary}

\section{Abelian consequences, and a comparison with earlier bounds}
\label{sec:comparison}

We now specialize once more, from an arbitrary locally compact group to a
locally compact \emph{abelian} group, to recover the classical Green--Sanders
theorem, and take the opportunity to place our bound alongside other
quantitative bounds available in the abelian literature.

Let $K$ be a locally compact abelian group and $\mu\in M(K)$ an idempotent
measure, meaning $\mu*\mu=\mu$. Taking Fourier transforms, this says
$\widehat\mu^2=\widehat\mu$ on the dual group $\widehat K$, i.e.\
$\widehat\mu$ is Boolean-valued. Under the standard isometric identification
$B(\widehat K)\cong M(K)$ (Fourier--Stieltjes transform), we have
$\|\widehat\mu\|_{B(\widehat K)}=\|\mu\|$. Applying Theorem~\ref{thm:group}
to the group $\widehat K$ and the integer-valued (indeed Boolean-valued)
function $\widehat\mu\in B(\widehat K)$ gives Corollary~\ref{cor:abelian} at once.

We next compare the quantitative bounds with the earlier results; the different
settings are not directly comparable in every respect.

\begin{itemize}
\item Green and Sanders's original theorem \cite{GS}, valid for arbitrary
locally compact abelian $K$, gave the qualitatively weaker bound $L\le
\exp\exp(C\|\mu\|^4)$: a tower of height two, rather than a single
exponential. Theorem~\ref{thm:group} removes one level of the tower, in a
setting that is in fact more general at the group level, since
Theorem~\ref{thm:group} allows arbitrary locally compact groups.

\item In the more restrictive setting of integer-valued functions on
\emph{finite abelian} groups, Sanders \cite{SandersA} summarizes his bound as
$\exp(M^{4+o(1)})$. An $\exp(O(M^4))$ estimate removes the $M^{o(1)}$
loss in the exponent. On the specific model group $\mathbb F_2^n$, however, Sanders's
$\exp(M^{3+o(1)})$ bound \cite{SandersF2} remains substantially sharper.
Our main point is that the same argument works without abelian-specific
input and gives a single-exponential bound for every locally compact group.

\item The Green--Sanders argument also keeps track of the number of
\emph{distinct} subgroups appearing across the whole decomposition, a finer
invariant than the total number $L$ of coset terms
(which allows repetitions). Our recursion, as it stands, only controls $L$;
different leaves of our recursion tree can and typically will produce
unrelated period subgroups, and we have no mechanism at present for
collapsing or coordinating them. We regard finding such a mechanism, in the
non-abelian setting, as an interesting question left open by this paper.

\item In the Boolean special case, Sanders's quantitative non-abelian
idempotent theorem \cite{SandersNA} for finite groups gives a triply iterated
tower bound in $O(\|f\|_{A(G)})$. Theorem~\ref{thm:group} replaces this by a
single exponential $\exp(O(\gamma^4))$, and, unlike \cite{SandersNA}, applies
verbatim to arbitrary locally compact groups, not just finite ones.

\item Sanders also obtained a related quantitative structural theorem for
Boolean functions on finite groups in terms of \emph{coset decision trees}
\cite{SandersCDT}.  If $G$ is finite and $f:G\to\{0,1\}$ satisfies
$\|f\|_{A(G)}\le M$,
then $f$ can be computed by a coset decision tree with at most
$\exp\!\bigl(\exp\!\bigl(\exp(O(M^2))\bigr)\bigr)$
leaves.  Conversely, Sanders observes that a Boolean function computed by a
coset decision tree with $m$ leaves has Fourier-algebra norm at most
$\exp(O(m))$.  Thus coset decision trees provide another quantitative
formulation of the principle that Boolean functions of small Fourier-algebra
norm are governed by coset structure.  This result controls a different
structural parameter from the signed coset-decomposition length considered
here, so the two bounds are not directly comparable.

\item For the matrix theorem, Beke, Goh, Hatami, Jaffe and Naylor
\cite{BGHJN} proved that a Boolean matrix of $\gamma_2$-norm at most
$\gamma$ is a signed sum of $2^{O(\gamma^6)}$ blocky matrices.
Corollary~\ref{cor:matrix} improves the exponent from six to four. The
integer-valued formulation is included for convenience; the generalized
almost-integer formulation in \cite{BGHJN} already encompasses integer
matrices as well.
\end{itemize}

\paragraph{Use of AI tools.}
AI-assisted tools were used solely to improve the wording, organization, and
clarity of the exposition. They were not used to generate the
mathematical arguments. 

\subsection*{Acknowledgments}
The author is grateful to Lianna Hambardzumyan for bringing this problem to her attention, to Daniel Naylor for helpful bibliographic suggestions, and to Imre Leader for his guidance and support. Part of this work was carried out during the Circuits, Communication, and Proofs program at ICTS. The author was supported by the CB European PhD Studentship, funded by Trinity College, Cambridge.

\end{document}